\documentclass[11pt]{article} 

\usepackage{geometry} 
\usepackage{booktabs} 
\usepackage{array} 
\usepackage{paralist} 
\usepackage{verbatim} 
\usepackage{subfig} 
\usepackage{amssymb,amsmath, amsthm} 
\usepackage{outlines}

\newcommand{\R}{\mathbb{R}}

\newcommand{\Pw}{\mathcal{P}}
\newcommand{\Inv}{\text{Inv}}

\usepackage{fancyhdr} 
\usepackage{sectsty}
\allsectionsfont{\sffamily\mdseries\upshape} 

\usepackage[nottoc,notlof,notlot]{tocbibind} 
\usepackage[titles,subfigure]{tocloft} 

\usepackage{pdfpages}
\usepackage{graphicx}
\usepackage{subfig}

\newtheorem{theorem}{Theorem}[section]
\newtheorem{corollary}{Corollary}[theorem]
\newtheorem{lemma}{Lemma}[section]

\newtheorem{definition}{Definition}[section]

\title{Isolating Neighborhoods and Filippov Systems: Extending Conley Theory to Differential Inclusions}
\author{Cameron Thieme}
\date{\today}

\begin{document}
\sloppy
\maketitle

\begin{abstract}
The Conley index theory is a powerful topological tool for obtaining information about invariant sets in continuous dynamical systems.  A key feature of Conley theory is that the index is robust under perturbation; given a continuous family of flows $\{\varphi_\lambda\}$, the index remains constant over a range of parameter values, avoiding many of the complications associated with bifurcations.  

This theory is well-developed for flows and homomorphisms, and has even been extended to certain classes of semiflows.  However, in recent years mathematicians and scientists have become interested in differential inclusions, often called Filippov systems, and Conley theory has not yet been fully developed in this setting.  This paper aims to begin extending the index theory to this larger class of dynamical systems.  In particular, we introduce a notion of perturbation that is suitable for this setting and show that isolating neighborhoods, a fundamental object in Conley index theory, are stable under this sense of perturbation.  We also discuss how perturbation in this sense allows us to provide rigorous results about "nearby smooth systems" by analyzing the discontinuous Filippov system.
\end{abstract}
\let\thefootnote\relax
\let\thefootnote\relax
\section{Introduction}

In dynamical systems we often consider differential equations of the form $\dot{x}=f(x)$, where $f$ is assumed to be smooth in some sense.  However, in many modern models, the underlying differential equations are not necessarily even continuous.  This set includes models of friction, where an object can reach a rest point in finite time, and models involving mechanical switching, where a solution evolves according to one vector field till it reaches a certain point and then switches abruptly to another \cite{bernardo}. Low dimensional climate models also frequently exhibit non-smooth behavior \cite{welander}.

In order to analyze these non-smooth systems, the discontinuous differential equation is frequently reframed as a \textit{differential inclusion} $$\dot{x}\in F(x)$$ where $F$ is a set-valued map.  Under basic assumptions on the set-valued map $F$, such a differential inclusion can be shown to have an absolutely continuous solution $x(\cdot)$ on some interval, satisfying $$\dot{x}(t)\in F(x(t))\hspace{2mm} a.e.,\hspace{1cm} x(0)=x_0$$ for any initial condition $x_0$ in the domain of $F$.  These solutions also have many other nice properties which we expect from dynamical systems.  These properties are documented in a seminal work by Filippov (\cite{filippov}), which is why such systems are frequently called \textit{Filippov systems}.  Since many readers may be unfamiliar with differential inclusions we will review some basic information on these systems in the next section.

Differential inclusions can be extremely complicated.  Because any standard differential equation can be rewritten as a differential inclusion, these systems feature all of the difficulties associated with classical dynamical systems, plus many of their own unique problems.  Chief among these novel complications is the fact that for a fixed initial condition $x_0$ there may be infinitely many different solutions to the differential inclusion.  Having infinitely many solutions prevents these dynamical systems from being described as a flow, a difficulty which will be addressed by the introduction of \textit{multiflows}, an object defined by Richard McGehee in order to generalize flows to this setting \cite{mcspeech}.   

Moreover, researchers frequently want to understand the relationship between the Filippov system and nearby smooth systems.  Here a nearby smooth system is some parametrized differential equation $$\dot{x}=f_\lambda(x)$$ where the smooth function $f_\lambda$ limits to a piecewise-continuous vector field $f$ as $\lambda$ approaches $0$.  The discontinuous vector field $f$ is then modified, generally using the Filippov convex combination method, to change it into a set-valued map $F$ that agrees with $f$ wherever $f$ is continuous.  This alteration allows us to study the differential inclusion $\dot{x}\in F(x)$.

The discontinuous model is often studied because it is simpler to understand, but it is not always clear whether behavior found in the Filippov system persists in nearby smooth systems \cite{jeffrey}.  Such issues motivate the study of the perturbation of differential inclusions.  We hope to extend Conley index theory to the setting of differential inclusions because Conley theory provides rigorous results about dynamical systems which are totally stable under perturbation.  That is, any information that we can obtain about the system $\dot{x}=f(x)$ using Conley theory will also hold under small perturbations of $f$, avoiding many of the complications of bifurcation theory.

The most basic setting for Conley index theory is in the context of flows $$\varphi:\R\times X\mapsto X$$ where $X$ is a locally compact metric space such as $X=\R^n$.
The main goal of the Conley index is to provide a robust qualitative description of compact invariant sets of such a flow.  A set $S\subset\R^n$ is said to be \textit{invariant} if $$\varphi(\R,S):=\cup_{t\in\R}\varphi(t,S)=S$$

In general, the dynamics of such an invariant set can be extremely complicated, and their structure can radically change under even small perturbations.  To obtain a detailed view of these dynamics requires a great deal of analysis.  However, the Conley index aims to give a rough picture of compact invariant sets in a much more computationally simple manner.  

The first object of interest in Conley theory is the \textit{isolating neighborhood}, which is a compact set whose maximal invariant set is contained in its interior.  That is, a compact set $N\subset\R^n$ is an isolating neighborhood for the flow $\varphi$ if $$\text{Inv}(N,\varphi):=\{x\in N|\varphi(t,x)\in N\, \forall t\in\R\}\subset\text{int}(N)$$  These neighborhoods allow us to study \textit{isolated invariant sets}, which are sets $S\subset\R^n$ satisfying $S=\text{Inv}(N,\varphi)$ for some isolating neighborhood $N$.  

The most important property of isolating neighborhoods is that they are stable under perturbation.  That is, given a continuous family of flows $\{\varphi_\lambda\}$, if $N$ is an isolating neighborhood for $\varphi_0$, then for $|\lambda|$ sufficiently small $N$ is also an isolating neighborhood for $\varphi_\lambda$.  It is this basic concept in Conley theory that we extend in this paper in theorem \ref{pert_thm}.  Similar results are proven in \cite{mrozek}, \cite{dzedzej} and \cite{casagrande}, but the context and perspective used in those paper are somewhat different from the one used here.  In addition, the result proven here applies more directly to a larger class of differential inclusions, removing any bounding terms on the set-valued vector field or compactness assumptions on its domain.  

This generalized result will come near the end of the paper because a fair bit of background is necessary to state and prove it in the setting of differential inclusions.  In fact, it is contained in the fifth of six sections.  The first section is this introduction, and the next is a review of some basic results from differential inclusions.  The third section extends the notion of perturbing a differential equation to the setting of differential inclusions; since this notion of perturbation is novel, we spend a good amount of space justifying our choice of definition.  Following that we introduce the multiflow, which allows us to discuss the solution sets of differential inclusions.  In the fifth section we discuss isolating neighborhoods and isolated invariant sets, finally arriving at the generalized result that isolating neighborhoods are stable under perturbation.  The final section is very brief, offering some conclusions and acknowledgements.   

Conley index theory tells us much more about dynamical systems than the simple result on isolating neighborhoods that we will generalize in this paper.  Using this theory, an invariant set can be decomposed into an attracting region, a repelling region, and a gradient-like region.  It also associates a homological index to the isolated invariant sets which can be computed using only the behaviour on the boundary of an isolating neighborhood for a single parameter value.  This index conveys robust information about the topology of the invariant sets.  For basic information on Conley theory, the reader is referred to the brief survey article \cite{mischaikow}.  For those interested in more detailed information on this topic, see Conley's manuscript \cite{conley} or the more complete survey \cite{mischmroz}; this paper attempts to emulate the phrasing of these ideas found in \cite{mischmroz} particularly.  These other features of Conley theory--the decomposition and the index--are not extended in this paper, but it is our ambition to do so in the future.  With this paper we hope to develop some framework and begin the process of extending Conley index theory to differential inclusions.

\section{The Basics of Differential Inclusions}\label{diffinc}
Before diving into the extension of Conley theory we will review some basic information about differential inclusions.  To do so we will first introduce some notation.

For simplicity, we will restrict our view to subsets of Euclidean space, and assume that $X$ is a subset of $\R^n$.  

Throughout this paper we will denote a closed $\varepsilon$ neighborhood of a set $S\subset X$ by $B_\varepsilon(S)$; that is, $$B_\varepsilon(S):=\{x\in X: |x-S|\leq\varepsilon\}$$ where $|x-S|:=\inf_{s\in S}|x-s|$.  We use this notation, which is generally reserved for $\varepsilon$-balls around single points, so that we can reserve the variable $N$ for isolating neighborhoods.  Note in particular that this neighborhood is closed.  We deal almost exclusively with closed neighborhoods in this setting and so it is cleaner to use notation which automatically conveys this notion.  

We will use the notation $\Pw(S)$ to denote the \textit{powerset} of $S$, or the set of all subsets of the set $S$.  Additionally, we will often be discussing set-valued maps in this paper, and it is useful to distinguish them from their single-valued counterparts.  In order to make that distinction obvious, all single-valued maps will be lower-case, like $f$, and set-valued maps will be capitalized, like $F$.  

Notice that a single-valued map $f$ can be considered as a set-valued map whose image is always a one-point set.  Because of this fact all results about set-valued maps are generalizations of results in the classic, single-valued case.  Viewing single-valued maps as one-point set-valued maps also allows us to make rigorous statements about perturbing a differential inclusion to a nearby smooth system.  

\subsection{Set-Valued Maps}

We now begin by defining \textit{upper-semicontinuity} for set-valued maps.  

\begin{definition}
Let $X$ and $Z$ be subsets of $\R^n$.  A set-valued function $F:X\to\Pw(Z)$ is said to be \textbf{upper-semicontinuous at the point x} if for any $\varepsilon>0$ there exists some $\delta>0$ such that $|x-y|<\delta$ implies that $F(y)\subset B_\varepsilon(F(x))$.
 
$F$ is said to be \textbf{upper-semicontinuous} if it is upper-semicontinuous at each $x\in X$.
\end{definition}

If we examine this definition, we notice that if an upper-semicontinuous set-valued map $F$ is in fact single-valued then it is continuous in the traditional sense.  Because of this fact, the definition of upper-semicontinuity provided here for set-valued maps contradicts the better known definition of upper-semicontinuity for real-valued functions, which are generally not continuous.  This contradictory terminology is somewhat unfortunate but standard, and so we will use it here.  

It is important to notice that the definition of upper-semicontinuity is inherently one-sided in the sense that $|x-y|<\delta$ does not imply that $F(x)\subset B_\varepsilon(F(y))$.  If $F$ did have such a symmetric property it would be called \textit{continuous}, but it turns out that that requirement is too stringent for many of the most important differential inclusions and so we will not assume it.  

The upper-semicontinuity of a set-valued function preserves many of the properties of continuity of single-valued functions.  Importantly, upper-semicontinuity on a compact domain implies boundedness.  A proof of this result may be found in Filippov's book, \cite{filippov}, or in the paper \cite{thieme}, which presents concise versions of the results from \cite{filippov} that are relevant to this paper.  

\begin{lemma}\cite{filippov}
Let $X$ and $Z$ be subsets of $\R^n$ and assume $X$ is compact.  If the set-valued function $F:X\to\Pw(Z)$ is upper-semicontinuous then it is bounded. 

In more specific terms, define $$|F(x)|:=\sup\{|v|:v\in F(x)\}$$ Then if $X$ is compact, $\sup_{x\in X}|F(x)|$ is finite.
\end{lemma}

On the other hand, there are many standard properties of continuous single-valued functions which do not extend to set-valued maps which are merely upper-semicontinuous.  For instance, even if the domain of an upper-semicontinuous function is compact, the upper-semicontinuity is not uniform; that is, the choice of $\delta$ in the definition generally depends on both $x$ and $\varepsilon$, even in a compact domain.  This property presents some difficulties in extending the notion of an isolating neighborhood, but they can be overcome.  

We can now list all of the properties of a set-valued map $F$ that we demand for a differential inclusion $\dot{x}\in F(x)$.  These requirements are described as the basic conditions by Filippov (\cite{filippov}), but we will give them a new name that reflects his role in this theory.

\begin{definition}
Let $X,Z$ be subsets of $\R^n$.  The set-valued map $F:X\to\Pw(Z)$ is said to satisfy the \textbf{Filippov Conditions} if it is upper-semicontinuous and if the set $F(x_0)$ is
\begin{itemize}
\item Compact
\item Convex
\item Non-empty
\end{itemize}
for each $x_0\in X$.  
\end{definition}

We note in particular that these conditions do not place any sort of global bounding term on the set-valued map $F$, and they allow the case that $X$ and $Z$ are not compact, distinguishing our work from previous expositions.  

\subsection{Discontinuous Vector fields and Differential Inclusions}

To begin this section let us record explicitly what is meant by a solution of the differential inclusion $\dot{x}\in F(x)$.


\begin{definition}
Let $G\subset\R^n$ be open and let $F:G\to\Pw(\R^n)$ be a set-valued map.  A solution of the differential inclusion $$\dot{x}\in F(x)$$ is an absolutely continuous function $$x:I\to\R^n$$ on some interval $I\subset \R$ whose derivative satisfies $$\dot{x}(t)\in F(x(t))$$ for almost all $t$ on the interval $I$.
\end{definition}

One commonly studied kind of differential inclusion comes from differential equations $\dot{x}=f(x)$ with piecewise continuous righthand side.  Beginning with this single-valued, discontinuous $f$, the Filippov convex combination method may be used to make $f$ set-valued at the points of discontinuity, and the resulting set-valued map meets the Filippov conditions.  This method is discussed extensively in \cite{filippov} and \cite{thieme}, but let us see a simple motivational example here.  


Let $G\subset\R^n$ be open.  We will begin with a piecewise continuous map $f:G\to\R^n$.  To define such a map in some generality assume $h:G\to\R$ is a continuous scalar function.  The function $h$ is used to partition the domain $G$ into two disjoint open regions, $$G_1=\{x\in G|h(x)>0\}\hspace{1cm} G_2=\{x\in G|h(x)<0\}$$ with a discontinuity boundary $\Sigma=\{x\in G|h(x)=0\}$.  For simplicity, assume that $f_1$ and $f_2$ are continuous on all of $G$ and let $$f(x)=
\begin{cases}
f_1 (x), & h(x)>0 \\
f_2 (x), & h(x)\leq 0
\end{cases}$$

We can use the Filippov convex combination method here to define a set-valued map $F$ which agrees with $f$ away from the discontinuity boundary $\Sigma$ and which satisfies the Filippov conditions.  In this simplified case where the domain $G$ has been split into only two regions by the function $h$, the Filippov convex combination method simply defines $F(x)$ to be the convex hull of the vectors $f_1(x)$ and $f_2(x)$ for all $x\in G$ such that $h(x)=0$.  Written completely, the resulting set-valued map is
$$F(x) =
\begin{cases}
f_1 (x), & h(x) >0 \\
f_2 (x), & h(x) < 0 \\
\{\alpha f_1(x) + (1-\alpha)f_2(x): \alpha\in[0,1]\} & h(x)=0
\end{cases}
$$

By the following theorem, which can be found in Filippov's book (\cite{filippov}) and the paper \cite{thieme}, $F$ satisfies the Filippov conditions.  In particular, notice that the only assumption placed on the maps $f_i$ is continuity, allowing us to study any piecewise continuous vector field, even if its growth is unbounded.

\begin{theorem}\label{convcomb}
If $F$ is the set-valued map created by using the Filippov convex combination method to modify a piecewise continuous single-valued function $f$, then $F$ satisfies the Filippov conditions.  
\end{theorem}

\begin{figure}[h]
\begin{center}
\includegraphics[scale=0.4]{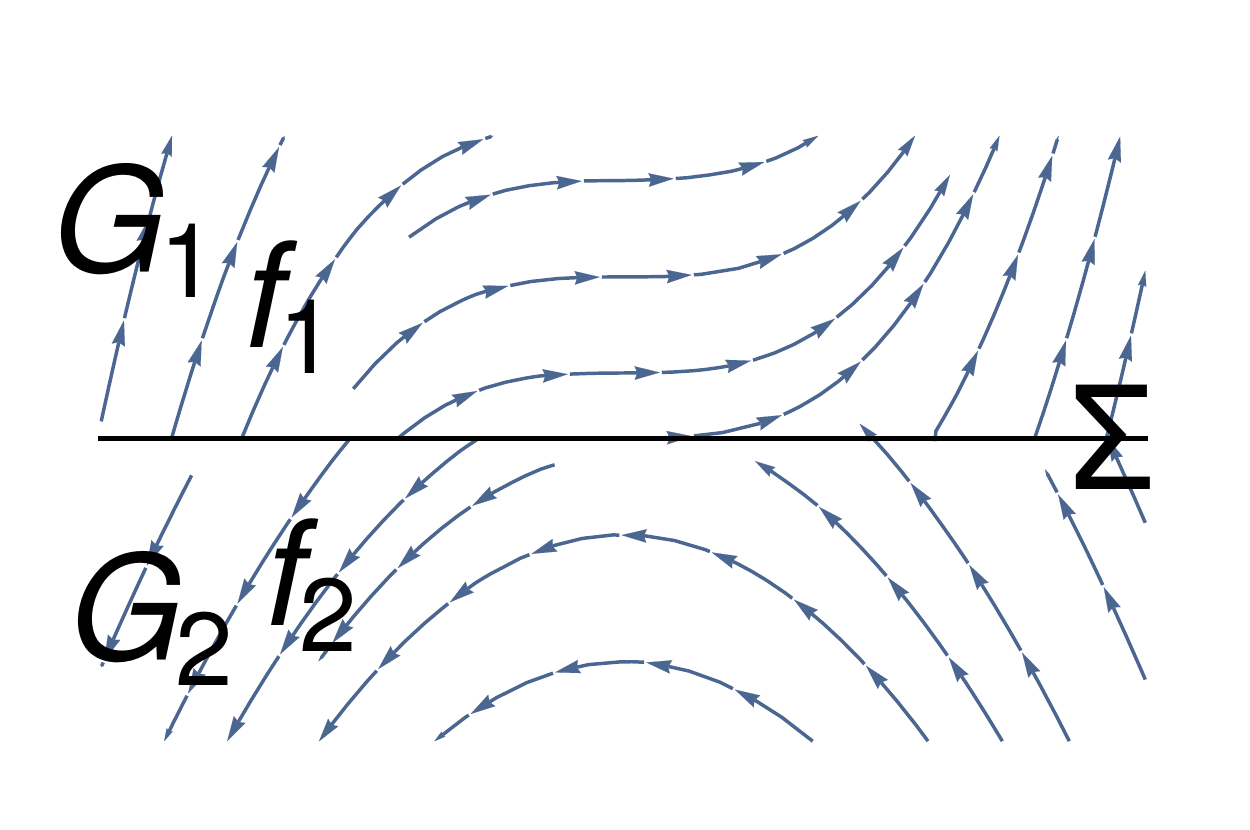}
\caption{A planar Filippov system with $\R^2$ split into two regions.}
\end{center}
\end{figure}

This method can be generalized to split any open domain $G\subset \R^n$ into some number of different subdomains, and the splitting boundary can be considerably more complicated.  We also note that many other systems---including ones which do not arise from piecewise continuous vector fields, such as control systems---also satisfy the Filippov conditions.  Systems like the example described above are our primary motivation for this paper, but the results extend to these other applications.  

\section{Perturbation of Differential Inclusions}

The ultimate goal of this paper is to generalize the result that isolating neighborhoods are stable under perturbation to the case of differential inclusions.  To do so, we need to define what we mean by perturbation in this setting.  For this definition to be interesting or meaningful, it should satisfy three qualitative conditions:

\begin{enumerate}
\item The definition should generalize standard notions of perturbation.  That is, in the special case where $F$ is single-valued, the new notion of perturbation must be equivalent to continuously perturbing a single-valued function.
\item The definition should allow us to say when set-valued maps are close to each other.\label{crit2}
\item The definition should allow us to perturb a differential inclusion to a nearby smooth system.
\end{enumerate}

The notion that we will adopt, which we will argue show satisfies all of these criteria, is an upper-semicontinuous family of differential inclusions, $\dot{x}\in F(x,\lambda)$, for $\lambda$ in some interval.  That is, for a given differential inclusion $\dot{x}\in F(x)$, we trivially write $F(x)=F(x,0)$ and consider a set-valued function $F(x,\lambda)$ that meets the Filippov conditions.  Then perturbing our original differential inclusion means that we perturb the parameter $\lambda$.  This notion of perturbation is very similar to the definition found in \cite{mrozek}, and the following definition records this idea.  

\begin{definition}
Let $G\subset\R^n$ and assume that $F:G\times[-1,1]\to\R^n$ meets the Filippov conditions; in particular, $F$ is upper-semicontinuous in both $x$ and $\lambda$ together.  Then the differential inclusion $$\dot{x}\in F(x,\lambda)$$ is considered to be a \textbf{$\lambda$-perturbation} of the differential inclusion $$\dot{x}\in F(x,0)$$
\end{definition}

Note that if $F$ were a single-valued function $f$, this sense of perturbation is equivalent to requiring that $f$ be continuous in both $x$ and $\lambda$.  Thus this novel definition of perturbation is, in fact, a generalization of the perturbation of a differential equation $\dot{x}=f(x)$, and so it meets the first, and most important, of our criteria for a new definition.  The remaining two criteria are somewhat more subjective than the first, and we will argue that this definition satisfies them in the following subsections.

\subsection{Perturbing a Set-Valued Function}

We return now to the example from section \ref{diffinc}:

$$\dot{x}\in F(x) =
\begin{cases}
f_1 (x), & h(x) >0 \\
f_2 (x), & h(x) < 0 \\
\{\alpha f_1(x) + (1-\alpha)f_2(x): \alpha\in[0,1]\} & h(x)=0
\end{cases}
$$

We will show that continuously perturbing either $h$ or the $f_i$ yields an allowable perturbation in our sense.  

\begin{lemma}
Let $G\subset\R^n$ be open and assume $f_i:G\times[-1,1]\to\R^n$ and $h:G\times[-1,1]\to\R^n$ are continuous.  Then the set-valued function $$F(x,\lambda) =
\begin{cases}
f_1 (x,\lambda), & h(x,\lambda) >0 \\
f_2 (x,\lambda), & h(x,\lambda) < 0 \\
\{\alpha f_1(x,\lambda) + (1-\alpha)f_2(x,\lambda): \alpha\in[0,1]\} & h(x,\lambda)=0
\end{cases}
$$
satisfies the Filippov conditions.
\end{lemma}

\begin{proof}
For a fixed $\lambda_0$, the fact that $F(\cdot,\lambda_0)$ is compact, convex, and non-empty valued is obvious.  Therefore we only need to check that $F$ is upper-semicontinuous at any $(x_0,\lambda_0)$ in the domain.  

Fix $(x_0,\lambda_0)$.  If $h(x_0,\lambda_0)\neq 0$ then by the continuity of $h$ we can choose $(x,\lambda)$ close enough to $(x_0,\lambda_0)$ that $h(x,\lambda)\neq 0$, and so the upper-semicontinuity of $F$ at $(x_0,\lambda_0)$ follows trivially from the continuity of $f_1$ and $f_2$.  Therefore we will assume that $h(x_0,\lambda_0)=0$.  

Fix $\varepsilon>0$.  By the continuity of the $f_i$, there exist $\delta_1>0$ and $\delta_2>0$ such that $$|(x_0,\lambda_0)-(x,\lambda)|<\delta_i$$ implies that $$|f_i(x_0,\lambda_0)-f_i(x,\lambda)|<\varepsilon$$ Define $\delta=\min(\delta_1,\delta_2)$ and arbitrarily choose $(x,\lambda)$ such that $|(x_0,\lambda_0)-(x,\lambda)|<\delta$.  In this case it is possible for $h(x,\lambda)$ to be positive, negative, or zero.  

We will first address the case that $h(x,\lambda)\neq0$, and assume without loss of generality that $h(x,\lambda)>0$.  Note that $F(x,\lambda)=\{f_1(x,\lambda)\}$ and that $f_1(x_0,\lambda_0)\in F(x_0,\lambda_0)$.  By our choice of $\delta$ we have that $|f_1(x_0,\lambda_0)-f_1(x,\lambda)|<\varepsilon$, and so $F(x,\lambda)\subset B_\varepsilon(F(x_0,\lambda_0))$.

The remaining possibility is that $h(x,\lambda)=0$, so $$F(x,\lambda)=\{\alpha f_1(x,\lambda) + (1-\alpha)f_2(x,\lambda): \alpha\in[0,1]\}$$ Choose any arbitrary vector $v\in F(x,\lambda)$.  Then $v=\alpha f_1(x,\lambda) + (1-\alpha)f_2(x,\lambda)$ for some fixed $\alpha\in[0,1]$.  Note that for this fixed $\alpha$, the vector $$v_0:=\alpha f_1(x_0,\lambda_0) + (1-\alpha)f_2(x_0,\lambda_0)$$ lies in $F(x_0,\lambda_0)$.  Then by our choice of $\delta$,
\begin{align*}
|v_0-v|&=|(\alpha f_1(x_0,\lambda_0) + (1-\alpha)f_2(x_0,\lambda_0))-(\alpha f_1(x,\lambda) + (1-\alpha)f_2(x,\lambda))|\\
&= |\alpha (f_1(x_0,\lambda_0)- f_1(x,\lambda))+ (1-\alpha)(f_2(x_0,\lambda_0)-f_2(x,\lambda))|\\
&\leq \alpha|f_1(x_0,\lambda_0)- f_1(x,\lambda)|+(1-\alpha)|f_2(x_0,\lambda_0)-f_2(x,\lambda)|\\
&\leq \alpha\varepsilon+(1-\alpha)\varepsilon\\
&=\varepsilon
\end{align*}
Thus $F(x,\lambda)\subset B_\varepsilon(F(x_0,\lambda_0))$ and so $F$ is upper-semicontinuous.
\end{proof}

\subsection{Perturbing to Nearby Smooth Systems}

Considering the third criteria---that we should be able to perturb a differential inclusion to a nearby smooth system---is somewhat more subtle.  To understand what we mean by a nearby smooth system we will return once again to the example of section \ref{diffinc}.  As before, we begin with an open domain $G\subset\R^n$ which is split by the manifold $\Sigma=\{x\in G\,|\,h(x)=0\}$.  This time, to consider a bit more generality, we will assume that our piecewise-continuous map $f$ is not defined on $\Sigma$.  Then we can write $$f(x)=
\begin{cases}
f_1 (x), & h(x)>0 \\
f_2 (x), & h(x)< 0
\end{cases}$$

The Filippov convex combination method can still be used here just as it was in section \ref{diffinc}.  However, in \cite{jeffrey}, Jeffrey has shown that there are actually infinitely many smooth families of functions $\{f_\lambda\}$ which limit pointwise to $f$ away from $\Sigma$.  Moreover, it is possible that the differential equations $$\dot{x}= f_\lambda(x)$$ actually have qualitatively different behavior from the differential inclusion $$\dot{x}\in F(x)$$ where the Filippov convex combination method is used to transform $f$ into $F$.  

For a simple and concrete example of what that means, we will consider the piecewise-continuous function $$f(x)=
\begin{cases}
1, & x<0 \\
3, & x> 0
\end{cases}$$
Define the functions $$f_\lambda(x) = \tanh(\frac{x}{\lambda})+2\hspace{1cm} g_\lambda(x)=\tanh(\frac{x}{\lambda})+2-2*\text{e}*\mu(\frac{x}{\lambda})$$
where $\mu$ is the smooth mollifier $$\mu(x)=
\begin{cases}
0, & |x|>1 \\
\exp(\frac{-1}{1-x^2}), & |x|\leq 1
\end{cases}$$

As $\lambda\to 0$, both $f_\lambda$ and $g_\lambda$ limit to $f$ pointwise on its domain $\R\setminus\{0\}$.  However, the dynamics of the differential equations $$\dot{x}=f_\lambda(x)\hspace{1cm}\dot{x}=g_\lambda(x)$$ are qualitatively different.  For all $\lambda$, $\dot{x}=f_\lambda(x)$ has no equilibria, but for any $\lambda$ there is an equilibrium of the ODE $\dot{x}=g_\lambda(x)$.  

If we use the convex combination method on $f$ we get the set-valued map $$F_0(x)=
\begin{cases}
1, & x<0 \\
3, & x> 0 \\
[1,3], & x=0
\end{cases}$$
Qualitatively, the differential inclusion $\dot{x}\in F_0(x)$ behaves like the differential equations $\dot{x}=f_\lambda(x)$, and so it seems natural to consider $f_\lambda$ to be a nearby smooth system for $F_0$.  Indeed, if we define the function $F:\R\times[0,1]\to\R$ by $$F(x,\lambda)=
\begin{cases}
F_0(x), & \lambda = 0 \\
f_\lambda(x), & \lambda> 0
\end{cases}$$
we can easily verify that $F$ satisfies the Filippov conditions.  Therefore $f_\lambda$ can be said to be a perturbation of $F_0$ in our sense.  

\begin{figure}[h]
\begin{center}
\includegraphics[scale=0.3]{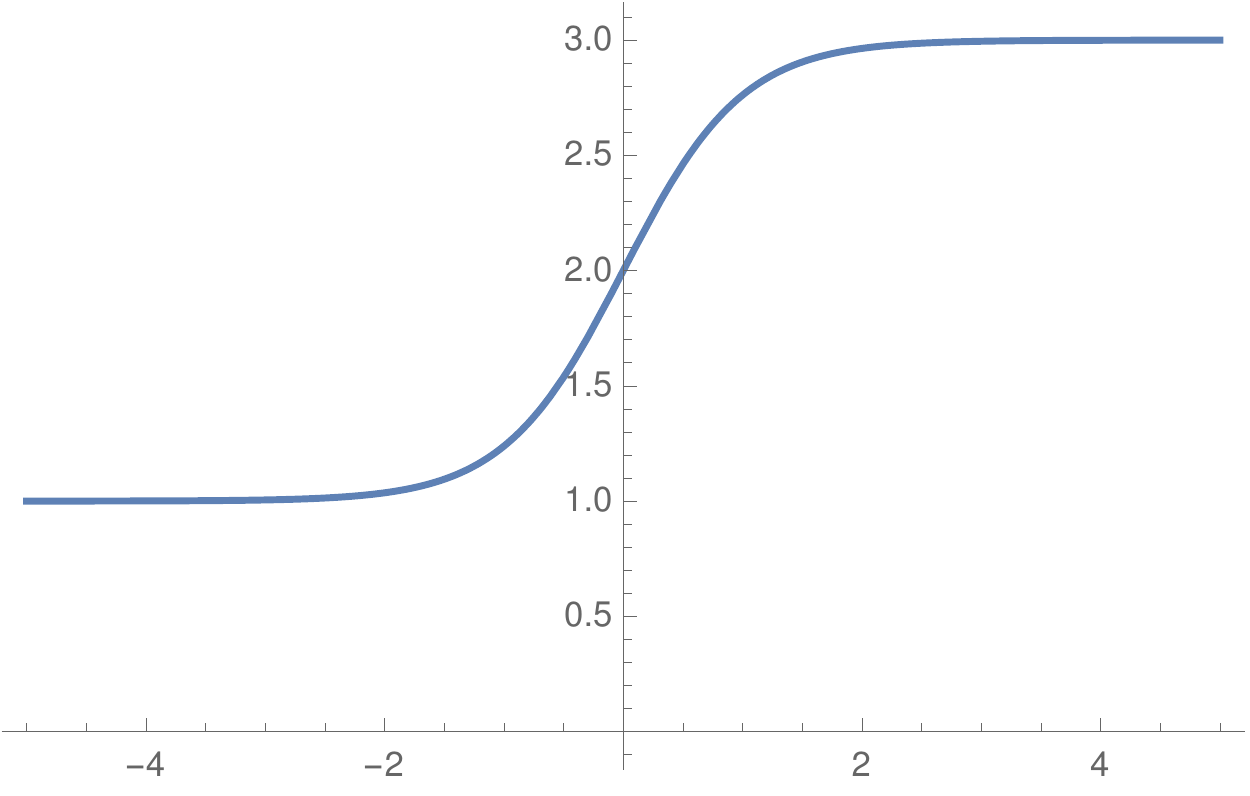}
\includegraphics[scale=0.3]{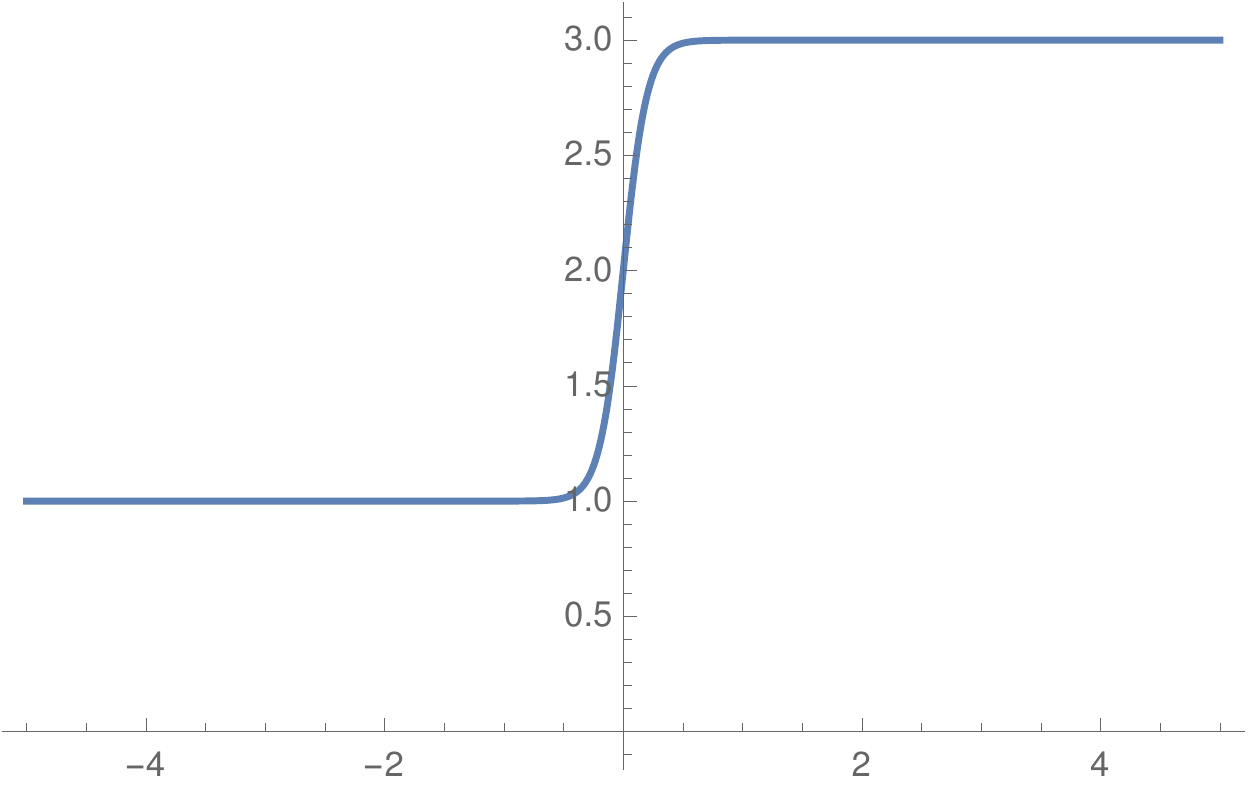}
\includegraphics[scale=0.3]{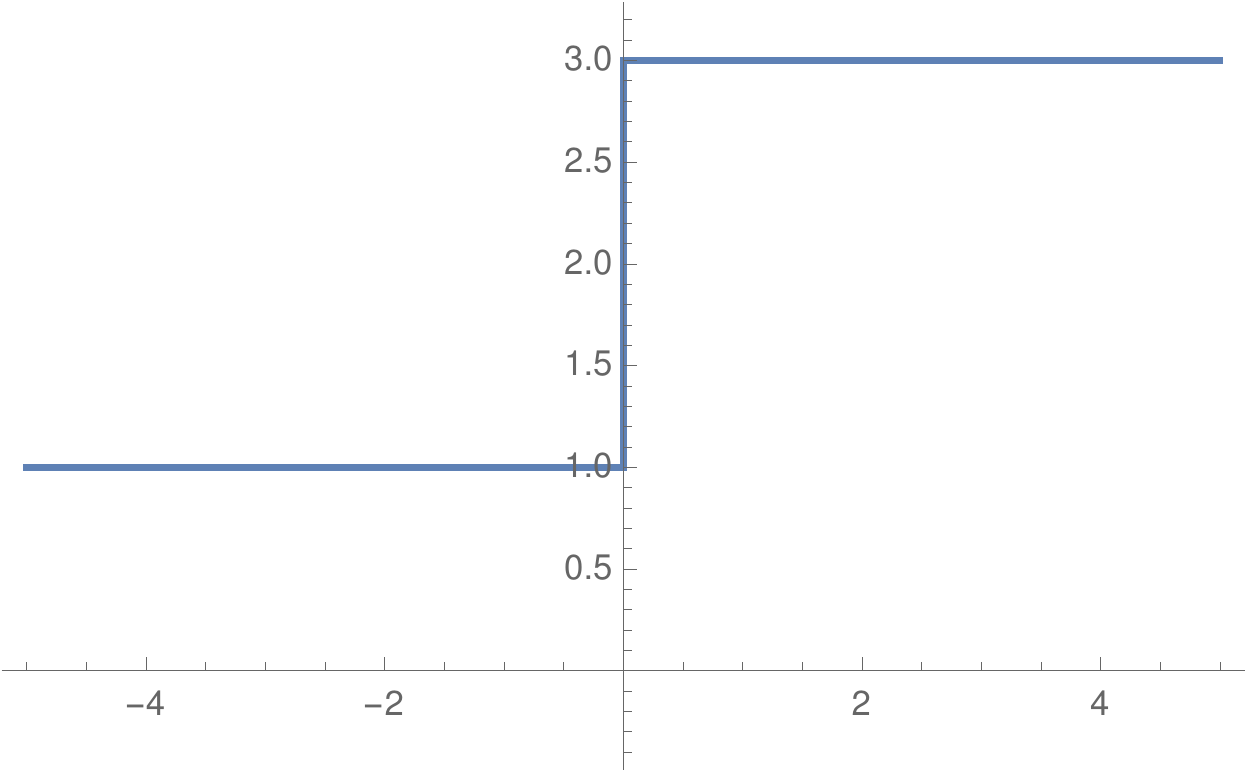}
\caption{The functions $f_1$, $f_{1/5}$, and $F_0$ respectively.}
\end{center}
\end{figure}

However, $g_\lambda$ is not an allowable perturbation of $F_0$; the function $$H(x,\lambda)=
\begin{cases}
F_0(x), & \lambda = 0 \\
g_\lambda(x), & \lambda> 0
\end{cases}$$ is not upper-semicontinuous.  

The more appropriate set-valued limit of $g_\lambda$ would be the function $$G_0(x)=
\begin{cases}
1, & x<0 \\
3, & x> 0 \\
[\tau,3], & x=0
\end{cases}$$
where $\tau=\min_{x\in\R}g_\lambda(x)<0$ (this minimum is independent of $\lambda$).  It is straightforward to verify that $G_0$ satisfies the Filippov conditions, and so we can reasonably consider the differential inclusion $\dot{x}\in G_0(x)$.  Moreover, note that $G_0$ preserves the qualitative features of $g_\lambda$; the function $\psi\equiv 0$ solves the differential inclusion $\dot{x}\in G_0(x)$, showing that this inclusion has an equilibrium.

Moreover, if we define the function $G:\R\times[0,1]\to\R$ by $$G(x,\lambda)=
\begin{cases}
G_0(x), & \lambda = 0 \\
g_\lambda(x), & \lambda> 0
\end{cases}$$
we can again verify that $G$ satisfies the Filippov conditions, and so $g_\lambda$ is an appropriate perturbation of $G_0$.

\begin{figure}[h]
\begin{center}
\includegraphics[scale=0.3]{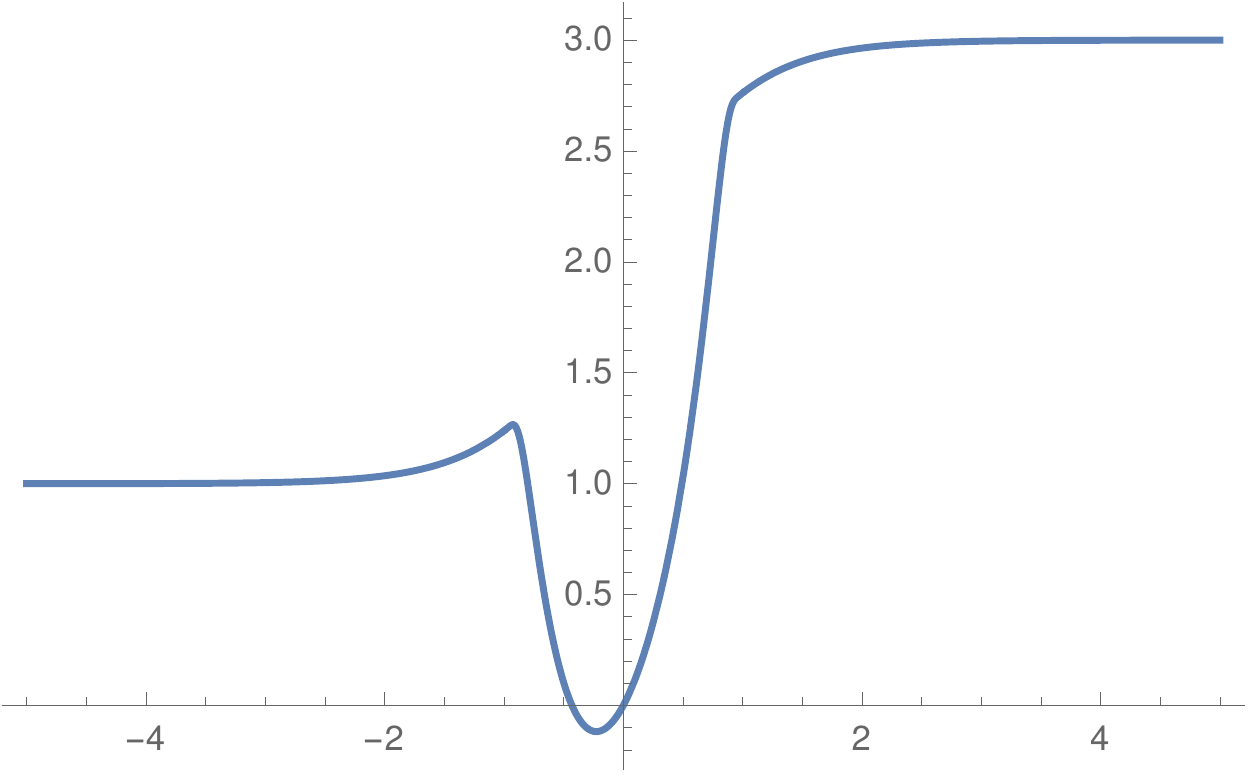}
\includegraphics[scale=0.3]{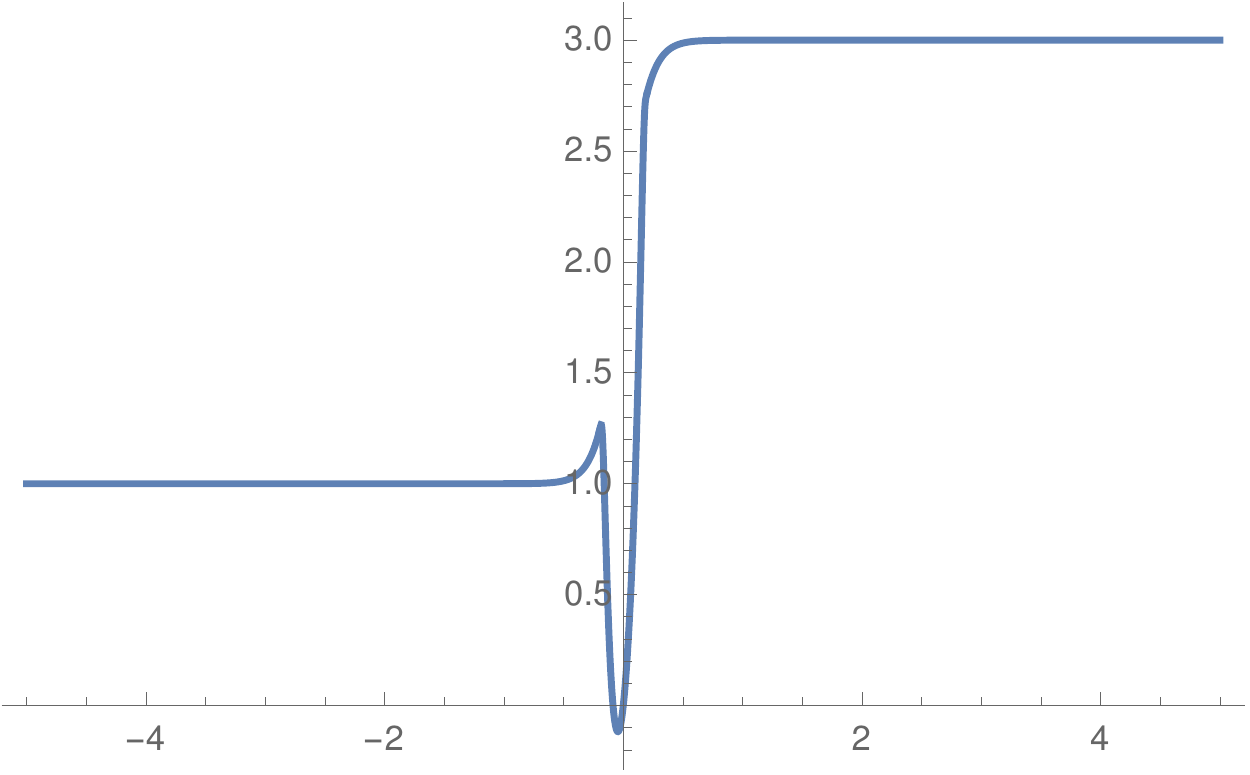}
\includegraphics[scale=0.3]{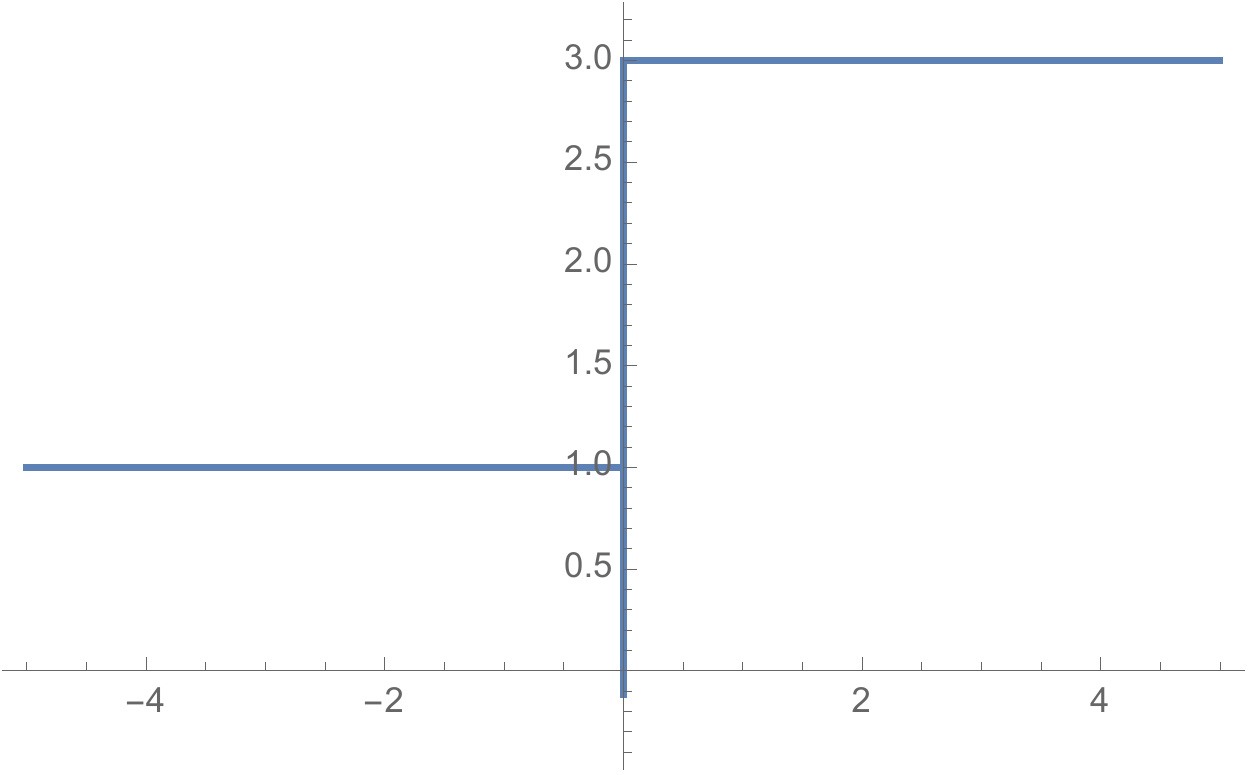}
\caption{The functions $g_1$, $g_{1/5}$, and $G_0$ respectively.}
\end{center}
\end{figure}

This example illustrates a general principle about these perturbations: piecewise-continuous systems may be the pointwise limit of a family of continuous systems, but that doesn't mean that family is an appropriate perturbation of the non-smooth system.  

Readers familiar with the Conley index may notice that in the prior example, if we take a neighborhood of $\Sigma=\{0\}$ and compute its index as we would in the smooth case (we note here that no proofs yet exist to claim that this index is meaningful in the case of differential inclusions, though we hope to change that) then the index is trivial for both $F_0$ and $G_0$.  That is, the Conley index does not distinguish between these systems even though our notion of perturbation does.  Therefore it is possible that the dynamics defined on $\Sigma$ may not effect the index (so long as the Filippov requirements are met) and a less restrictive notion of perturbation might exist \cite{mcspeech}.  However, what such a definition might be is unclear at this time.  

What this definition of perturbation does provide is a simple way to check whether a limiting smooth system is close to a given differential inclusion.  Since single-valued functions are automatically compact, convex, and non-empty valued (at each point) checking the Filippov conditions simply requires checking upper-semicontinuity, a task which is about as easy as checking continuity in the classical case.

\subsection{Perturbed Solutions}

When we study a perturbed differential inclusion we are of course interested in the associated perturbed solutions.  Fortunately for us, Filippov developed a notion of an approximate $\delta$-solutions to the differential inclusion $\dot{x}\in F(x)$ that plays quite nicely with our notion of perturbation.
This definition relies on the notation that for a set $A$, $co(A)$ denotes the smallest convex subset containing $A$. 

\begin{definition}
A \textbf{$\delta$-solution} of the differential inclusion $\dot{x}\in F(x)$ is an absolutely continuous function $y(\cdot)$ that almost everywhere satisfies the differential inclusion 
$$\dot{y}(t)\in F_\delta(y(t))$$
where $F_\delta(y):=B_\delta(co(F(B_\delta(y))))$
\end{definition}

One of Filippov's lemmas says that a uniformly convergent sequence of approximate solutions converges to an exact solution, assuming that the approximating radius also converges to zero:

\begin{lemma}\label{unisol}\cite{filippov}\cite{thieme}
Let $F(x)$ satisfy the basic conditions in a domain $G\subset\R^n$ and let $\delta_k\to 0$ as $k\to\infty$.  Then the limit $x(\cdot)$ of a uniformly convergent sequence $\{x_k:[a,b]\to G\}$ of $\delta_k$-solutions  to the differential inclusion $\dot{x}\in F(x)$ is a solution to that inclusion (as long as $x(t)\in G$).
\end{lemma}

These approximate solutions are very helpful because, in some sense, they give us the symmetry that upper-semicontinuity lacks.  In essence, we are saved because $F_\delta(y)$ first takes a $\delta$ neighborhood of $y$, then applies $F$, then convexifies, and then applies another $\delta$ neighborhood to the resulting image set.  The fact that we first consider a neighborhood around $y$, and then apply $F$, allows us to overcome the lack of uniformity in the upper-semicontinuity.  As an aside, we note that the convexification is not used directly in this paper, but it is needed so that the above lemma is true.

In order to utilize these approximate solutions we will show that perturbed solutions are approximate solutions.  In the following lemmas we will assume that our domain is compact.  This assumption is standard in Conley theory, and it will not limit our later results about isolating neighborhoods.

\begin{lemma}\label{pertappx}
Assume $X\subset\R^n$ is compact.  If $F:X\times[-1,1]\to\R^n$ satisfies the Filippov conditions, then for each $\varepsilon>0$ there exists a $\delta>0$ such that $|\lambda|<\delta$ implies that $$F(x,\lambda)\subset F_\varepsilon(x,0):=B_\varepsilon(co(F(B_\varepsilon(x),0)))$$
\end{lemma}

\begin{proof}
Choose $\varepsilon>0$.  Since $F$ is upper-semicontinuous in both $x$ and $\lambda$, for each $(x,0)$ in the subspace $X\times \{0\}$ there is some $\delta_x'$ such that
$$|(x,0)-(y,\lambda)|<\delta_x'\implies F(y,\lambda)\subset B_\varepsilon(F(x,0))$$
Let $\delta_x:=\min(\frac{\delta_x'}{2},\varepsilon)$. 
The sets $\{\text{int}(B_{\delta_x}(x))\}_{x\in X}$ cover $X$, and since the space $X$ is assumed to be compact we can find some finite subcover $\{\text{int}(B_{\delta_{x_k}}(x_k))\}_{k=1}^m$.  Let $$\delta:=\min_{1\leq k\leq m}\delta_{x_k}$$
The basic open sets $$\{\text{int}(B_{\delta_{x_k}}(x_k))\times (-\delta,\delta)\}_{k=1}^m$$ form an open cover of $X\times(-\delta,\delta)$, and so for any $(x,\lambda)\in X\times(-\delta,\delta)$, there is some $k$ such that 
$$|(x_k,0)-(x,\lambda)|\leq |(x_k,0)-(x,0)|+|(x,0)-(x,\lambda)|<\delta_{x_k}+\delta \leq \delta_{x_k}+\delta_{x_k}\leq\delta_{x_k}'$$
Thus $F(x,\lambda)\in B_\varepsilon(F(x_k,0))$.

Since we have assumed that $\delta_{x_k}\leq\varepsilon$, we have that $|x-x_k|<\varepsilon$.  Therefore 
$$ F(x,\lambda)\subset B_\varepsilon(F(x_k,0)) = B_\varepsilon(co(F(x_k,0)))\subset B_\varepsilon(co(F(B_\varepsilon(x),0)))=F_\varepsilon(x,0)  $$
\end{proof}

Our purpose in proving that lemma is to get the following corollary, which says that solutions to a perturbed differential inclusion are approximate solutions to the original differential inclusion.  

\begin{corollary}
If $F(x,\lambda)$ satisfies the Filippov conditions, then for each $\varepsilon>0$ there exists some $\delta>0$ such that $|\lambda|<\delta$ implies that solutions to the perturbed differential inclusion $\dot{x}\in F(x,\lambda)$ are $\varepsilon$-solutions of the differential inclusion $\dot{x}\in F(x,0)$.  
\end{corollary} 

\begin{proof}

A solution to the perturbed differential inclusion $\dot{x}\in F(x,\lambda)$ is an absolutely continuous function $\psi$ mapping from an interval of the real line into the state space which satisfies $$\dot{\psi}(t)\in F(\psi(t),\lambda)$$ almost everywhere.
  
By lemma \ref{pertappx}, for each $\varepsilon>0$ there exists $\delta>0$ such that $|\lambda|<\delta$ implies that $F(x,\lambda)\subset F_\varepsilon(x,0)$ for all $x$.  Then for such $\lambda$, $$\dot{\psi}(t)\in F(\psi(t),\lambda)\subset F_\varepsilon(\psi(t),0)$$ wherever the derivative exists (almost everywhere).

\end{proof}

We will also state the following simple lemma, which is needed for the proof of our main result and may be interesting in its own right.

\begin{lemma}\label{lipsch}
Assume $F$ satisfies the Filippov conditions and fix $\delta>0$.  For any $\varepsilon$ such that $0<\varepsilon<\delta$, all $\varepsilon$-solutions to the differential inclusion $\dot{x}\in F(x)$ which lie in a compact domain $X$ share a uniform Lipschitz bound (which does not depend on $\varepsilon$).  Thus $\varepsilon$-solutions are uniformly equicontinuous irrespective of $\varepsilon$.   
\end{lemma}

\begin{proof}

Since upper-semicontinuous functions are bounded on compact domains, there is some $M>0$ such that $|F(x)|<M$ for all $x\in X$.  Note that $|F_\varepsilon(x)|\leq M+\delta$.  Then for any $\varepsilon$-solution $\psi: I\to X$,
$$|\psi(t)-\psi(s)|=|\int_s^t\dot{\psi}(\tau)d\tau|\leq|\int_s^t |\dot{\psi}(\tau)|d\tau|\leq|\int_s^t (M+\delta)\, d\tau|=(M+\delta)|t-s|$$
\end{proof}

Because we are interested in discussing invariant sets, we are naturally interested in solutions which are defined for all time.  Since lemma \ref{unisol} only applies to solutions defined on a compact interval we will extend it with the following simple lemma.

\begin{lemma}\label{extunisol}
Assume that $X\subset\R^n$ is compact and that $F:X\to\R^n$ satisfies the Filippov conditions.  Let $\delta_k\to 0$ as $k\to\infty$.  Given any sequence $\{x_k:\R\to X\}_{k=1}^\infty$ of $\delta_k$-solutions  to the differential inclusion $\dot{x}\in F(x)$, there is an exact solution $$x:\R\to X$$ to that inclusion such that on any compact interval $[a,b]\subset\R$, there is a subsequence of the restricted family $$\{x_k|_{[a,b]}:[a,b]\to X\}_{k=1}^\infty$$ which converges uniformly to $x|_{[a,b]}$.
\end{lemma}

\begin{proof}
The idea of this proof is relatively straightforward but may be obfuscated by indexing, and so we will first sketch the basic intuition.  We consider a nested sequence of compact intervals that grow to encompass all of $\R$.  On each interval the Arzela-Ascoli theorem gives us a uniformly convergent subsequence that limits to a continuous function defined on that compact interval; by lemma \ref{unisol}, that function is a solution to $\dot{x}\in F(x)$.  In successive steps we take subsequences of the prior subsequences so that the limit function defined on the larger interval agrees with the limit function on any smaller interval.  Then we use these limit functions in order to pointwise define the function $x(\cdot)$ on all of $\R$ that satisfies the desired requirements.

In more detail, the proof goes as follows:

First, note that all $x_k(\cdot)$ are uniformly Lipschitz by lemma \ref{lipsch}.  More to the point, they are uniformly equicontinuous.

For all $k$, the function $x_k$ is defined on the interval $[-1,1]$, and so we can consider the restricted family $\{x_k|_{[-1,1]}:[-1,1]\to X\}_{k=1}^\infty$.  Since this family of functions is uniformly bounded (contained in $X$) and equicontinuous, by the Arzela-Ascoli theorem there is some subsequence $\{x_{k_i}|_{[-1,1]}:[-1,1]\to X\}_{i=1}^\infty$ which converges uniformly to a continuous function $x^1:[-1,1]\to X$.  By lemma \ref{unisol}, $x^1$ is a solution of the differential inclusion $\dot{x}\in F(x)$.  

In the next step, we note that $x_{k_i}$ is defined on the interval $[-2,2]$ for all $i$, and so we may consider the uniformly bounded and equicontinuous family $\{x_{k_i}|_{[-2,2]}:[-2,2]\to X\}_{i=1}^\infty$.  As in the last step, this has some subsequence which converges to a solution $x^2:[-2,2]\to X$.  Here we note that since $x_{k_i}(t)\to x^1(t)$ for $t\in[-1,1]$, we have that $x^2|_{[-1,1]}\equiv x^1$.

We iterate this process inductively.  That is, at step $m$ we use the subsequence defined in step $(m-1)$ and then pass to another subsequence that converges on the larger interval $[-m,m]$.  In this way we get subsequences of $\{x_k|_{[-m,m]}:[-m,m]\to X\}_{k=1}^\infty$ which converge uniformly to continuous functions $x^m:[-m,m]\to X$ and satisfy $$x^m|_{[-q,q]}\equiv x^q$$ for any $q\leq m$.  By lemma \ref{unisol}, each of these functions is a solution to the differential inclusion on its domain.

Finally, we pointwise define the function $x(t):=x^m(t)$, where here $m$ is the least integer greater than or equal to $t$.  To see that for any compact interval $[a,b]$ there is a subsequence of the restricted family $\{x_k|_{[a,b]}:[a,b]\to X\}_{k=1}^\infty$ which converges uniformly to $x|_{[a,b]}$, simply take $m$ large enough that $[a,b]\subset[-m,m]$ and use the subsequence from step $m$ of the induction.  

By lemma \ref{unisol}, $x|_{[a,b]}$ is a solution for the interval $[a,b]$.  Since all solutions share a uniform Lipschitz bound we see that $x$ is Lipschitz continuous on all of $\R$ and thus absolutely continuous.  Then on any interval $[-m,m]$, for almost all $t\in [-m,m]$, the function $x(\cdot)$ is differentiable at $t$ and 
$$\dot{x}(t)=\dot{x}|_{[-m,m]}(t)\in F(x|_{[-m,m]}(t))=F(x(t))$$

For any $t\in \R$ we can choose $m$ large enough that $t\in[-m,m]$, and so $x$ is an exact solution to the differential inclusion $\dot{x}\in F(x)$ on all of $\R$.  
\end{proof}

The lemmas stated and proven in this section provide all of the information about solutions to differential inclusions that we will need for our purposes.  Before we can discuss isolating neighborhoods, however, we need some way to organize the full solution set of a differential inclusion.  

\section{Multiflows: A Generalization of Flows for Filippov Systems}

In dynamical systems we are usually concerned with the full solution set of an ordinary differential equation $\dot{x}= f(x)$. If $f$ is globally Lipschitz then this set can be neatly described as a flow.  A flow on a metric space $X$ is a continuous map $\varphi : \R \times X \mapsto X$ such that
\begin{enumerate}
\item $\varphi(0, x_0 ) = x_0$
\item $\varphi(s, \varphi(t, x_0 )) = \varphi(s + t, x_0 )$
\end{enumerate}

The differential equation $\dot{x}=f(x)$ gives rise to a flow by letting $\varphi(t, x_0 )$ be the solution $x(t)$ with the initial condition $x(0) = x_0$.  If $f$ is only locally Lipschitz then the differential equation gives rise to a similar object, a \textit{local flow}, whose domain is an open subset of $\R\times X$.  

Since differential inclusions do not have unique solutions---a given initial condition may be sent to infinitely many different locations at a fixed time $t$---we cannot study them with single-valued maps like flows.  Additionally, the group action from $\R$ must be reduced to a monoid action from $\R^+:=\{t\in\R|t\geq 0\}$ because trajectories may intersect \cite{thieme}.  To deal with these complications, Richard McGehee has proposed a different object, the \textit{multiflow}, which generalizes flows to this setting \cite{mcspeech}.  For a more complete exposition on this subject (and differential inclusions in general), the reader is referred to an earlier paper \cite{thieme}.

\begin{definition}
A \textbf{multiflow} on a metric space $X$ is a set-valued map $$\Phi:\R^+\times X\to\Pw(X)$$ which is upper-semicontinuous and compact-valued, and which satisfies the monoid properties:
\begin{itemize}
\item $\Phi^0(x)=\{x\}$
\item $\Phi^t(\Phi^s(x))=\Phi^{t+s}(x)$
\end{itemize}
\end{definition}

The above definition relies on the notation that for $A\subset X$, we write 
$$\Phi^t(\cdot)=\Phi(t,\cdot)\, \text{and}\,\Phi^t(A)=\cup_{x\in A}\Phi^t(x)$$
Additionally, if $I\subset\R^+$, we write
$$\Phi(I,A)=\Phi^I(A)=\cup_{t\in I}\cup_{x\in A}\Phi(t,x)$$

One interesting property of multiflows is that it is entirely possible that $\Phi(t,x)=\emptyset$ for some $(t,x)\in\R^+\times X$.  This feature is both an advantage and a complication in this setting.  One advantage because there is no need to talk about anything like a local flow.  For example, the differential equation $\dot{x}=x^2$ gives only a local flow because of finite time blowup.  However, the set of all solutions to this differential equation forms a multiflow over $\R$, where $\Phi(t,x)=\emptyset$ for any $t$ after the finite time blowup of $x$.  The disadvantage, however, is that the empty set makes putting a topology on multiflows more complicated.  

But the main advantage of allowing $\Phi$ to map into the empty set is that it allows us to more easily examine solutions only until they leave a compact domain.  That is, for a compact set $X$, we can consider the set of all solutions $$\{x:[0,T]\to X\,|\,T\in\R^+,\,\dot{x}(t)\in F(x(t))\,\text{a.e.}\}$$ and ignore the extension of these solutions outside of this codomain.  This perspective is natural for Conley index theory, where computation of the index relies only on a compact isolating neighborhood $N$, allowing us to avoid some of the complications of global dynamics.  It is this perspective shift that primarily distinguishes our work on multiflows from the work done by Mrozek in \cite{mrozek}, where the Conley index theory has also been applied to differential inclusions.  

With this perspective we can show that any differential inclusion which satisfies the Filippov conditions gives rise to a multiflow on a compact space.  Let $\dot{x}\in F(x)$ be any Filippov system defined on an open domain $G\subset\R^n$, and let $X$ be any compact subset of $G$. We are interested in considering all solutions of the differential inclusion which are contained entirely in $X$.  If a solution begins in $X$, but then leaves, we only monitor the solution up until the time it leaves.  We want to show that the union of all of these solutions forms a multiflow.
 
More explicitly, define $\Phi:\R^+\times X\to\Pw(X)$ by saying that $b\in\Phi(T,a)$ if and only if there exists a solution $x:[0,T]\to X$ to the differential inclusion $\dot{x}\in F(X)$ with $x(0)=a$ and $x(T)=b$.  Then the set $\Phi$ is a multiflow over $X$; a proof of this result can be found in \cite{thieme}.  Because of its importance to our paper, we will state this result here as a theorem.  

\begin{theorem}\label{mtflw}\cite{thieme}
Let $G$ be an open subset of $\R^n$, let $X\subset G$ be compact, and let $F:G\to\Pw(\R^n)$ satisfy the Filippov conditions.  Define the set-valued map $$\Phi:\R^+\times X\to \Pw(X)$$ by saying that $b\in\Phi(T,a)$ if and only if there exists a solution $x:[0,T]\to X$ to the differential inclusion $\dot{x}\in F(X)$ with $x(0)=a$ and $x(T)=b$.

Then $\Phi$ is a multiflow over $X$.
\end{theorem}

Notice that this theorem is very general; any continuous, single-valued function $f$ satisfies the Filippov requirements.  Therefore multiflows gives us a framework to discuss Conley theory on compact sets very directly for an extremely large class of dynamical systems.  

One unfortunate drawback to multiflows is that, like semiflows, they consider only forward time.  This feature is a necessity; it is impossible to retain the group action of all of $\R$ \cite{thieme}.  So that we may consider solutions in backwards time we will introduce the dual multiflow $\Phi^*:\R^-\times X\times X$, which we define pointwise by $$\Phi^*(T,a):=\{b\in X|a\in\Phi(-T,b)\}$$

Essentially, we may think of this object as the backwards time equivalent of the multiflow $\Phi$.  If $\Phi$ arises from the differential inclusion $\dot{x}\in F(x)$, as in theorem \ref{mtflw}, then it is straightforward to verify that $\Phi^*$ is the set of all solutions to the same differential inclusion in backwards time.  That is, $$b\in\Phi^*(T,a)$$ if and only if there is a solution $$x:[T,0]\to X$$ satisfying $x(0)=a$ and $x(T)=b$.  

Finally, if the multiflow $\Phi$ arises from the differential inclusion $\dot{x}\in F(x)$ then we will call solutions of that differential inclusion \textit{orbits} on $\Phi$.  Note that if an orbit $\psi$ has a domain that includes both positive and negative times then its image lies in both $\Phi$ and $\Phi^*$.  To be more specific about what that means, assume that $I\subset\R$ is an interval around $0$ and $\psi:I\to X$ satisfies $\dot{\psi}(t)\in F(\psi(t))$ for almost all $t\in I$.  If $T\in I$ is positive then $$\psi(T)\in\Phi(T,\psi(0))$$ and if $T\in I$ is negative then $$\psi(T)\in\Phi^*(T,\psi(0))$$

\section{Isolating Neighborhoods and Isolated Invariant Sets}

We are now finally prepared to begin discussing the notions of invariant sets and isolating neighborhoods.

\begin{definition}
For a multiflow $\Phi:\R^+\times X\to \Pw(X)$ on a compact subset $X\subset\R^n$, a set $S\subset X$ is called \textbf{invariant} under $\Phi$ if for each $x\in S$ there exists an orbit $$\psi:\R\to S$$ on $\Phi$ with $\psi(0)=x$.
\end{definition}

We note that in the above definition, it is implicit that for $t<0$, $\psi(t)\in\Phi^*(-t,\psi(0))$, and for $t>0$, $\psi(t)\in\Phi(t,\psi(0))$.

For our purposes, the crucial idea here is that if the multiflow $\Phi$ is the solution set to a differential inclusions $\dot{x}\in F(x)$ then $S$ is invariant  if and only if there is some solution $\psi(\cdot)$ which remains in $S$ for all time.  

For those more acquainted with flows and differential equations, it might seem that the correct notion of invariance would be that $\Phi(\R^+,S):=\cup_{t\in\R^+}\Phi(t,S)=S$ and $\Phi^*(\R,S):=\cup_{t\in\R^+}\Phi^*(t,S)=S$.  Note that if $\Phi$ is single-valued (for instance in the case where the multiflow arises from a Lipschitz differential equation) then these definitions are equivalent, and so this definition is really an extension of the ordinary notion of invariance for flows.  However, in this setting, the condition that $\Phi(\R^+,S)=S$ turns out to be too strong.  One reason that it is too strong is that the $\omega$-limit set of a set $S$,
$$\omega(S):=\cap_{T\geq 0}\overline{\cup_{t\geq T}\Phi^t(S)}$$ is not invariant under this stricter definition.  For an example of this behavior, consider the multiflow $\Phi$ which arises from the differential inclusion $\dot{x}\in F(x)$, where $F:\R\to\R$ is defined by $$F(x)=\begin{cases}
-x, & x<0\\
[0,1], & x = 0\\
1, & x>0
\end{cases}$$ Then $\omega(x)=0$ for $x<0$, but $\Phi(t,0)=[0,t]$ for all $t>0$. 

We now continue with a few more definitions.  

\begin{definition}
For a set $N\in X$ and a given multiflow $\Phi$ on $X$, we denote the maximal invariant set contained in $N$ by $\text{Inv}(N,\Phi)$; that is,
$$\text{Inv}(N,\Phi):=\{x\in N| \exists\, \text{orbit}\, \psi:\R\to N, \psi(0)=x\}$$
\end{definition}

\begin{definition}
A compact set $N\subset X$ is called an \textbf{isolating neighborhood} if its maximal invariant set lies in its interior; that is,
$$\text{Inv}(N)\cap\partial N=\emptyset$$
\end{definition}

With these definitions in hand, along with all of the lemmas from the prior section, we are finally ready to state and prove the main result of this paper.  

\begin{theorem}\label{pert_thm}
Let $X\subset\R^n$ be compact and assume that $F:X\times[-1,1]\to\Pw(\R^n)$ satisfies the Filippov conditions.  For each $\lambda\in[-1,1]$, define the multiflow $$\Phi_\lambda:\R^+\times X\to \Pw(X)$$ by saying that $b\in\Phi(T,a)$ if and only if there is a solution $x:[0,T]\to X$ of the differential inclusion $\dot{x}\in F(x,\lambda)$ such that $x(0)=a$ and $x(T)=b$.

If $N$ is an isolating neighborhood for the multiflow $\Phi_0$ then there exists some $\varepsilon>0$ such that $|\lambda|<\varepsilon$ implies that $N$ is an isolating neighborhood for $\Phi_\lambda$.  
\end{theorem}

\begin{proof}

We proceed by contradiction.  Suppose that for each $\varepsilon>0$ there was some $\lambda$ with $|\lambda|<\varepsilon$ such that $N$ was not an isolating neighborhood for $\Phi_\lambda$.  Then we can make sequences $$\{\varepsilon_k\}_{k=1}^\infty,\hspace{1cm}\{\lambda_k\}_{k=1}^\infty$$ such that $\varepsilon_k\to 0$ as $k\to\infty$, $|\lambda_k|<\varepsilon_k$, and $N$ is not an isolating neighborhood for $\Phi_{\lambda_k}$.   Since $N$ is not an isolating neighborhood for $\Phi_{\lambda_k}$, for each $k$ there is some orbit $$\psi_k:\R\to N$$ of $\Phi_{\lambda_k}$ such that $\psi_k(0)\in\partial N$.  

The sequence $\{\psi_k\}_{k=1}^\infty$ is uniformly bounded (since the range of each orbit is the compact space $X$) and it is equicontinuous by lemma \ref{lipsch}.  Therefore by lemma \ref{extunisol} there is an orbit $$\psi:\R\to N$$ of $\Phi_0$ such that on any compact interval $I$, there is a subsequence of $\{\psi_k|_I\}_{k=1}^\infty$ that converges uniformly to $\psi$.  Since $\partial N$ is compact and $\psi_k(0)\in\partial N$ for each $k$, $\psi(0)\in\partial N$.  This contradicts our assumption that $N$ is an isolating neighborhood for $\Phi_0$.  


\end{proof}

The prior result gives some indication that Conley theory can be extended to differential inclusions.  This next result is more trivial, but is good to check in order to further that goal.  This result says that isolated invariant sets, another important object in Conley theory, are compact.  

\begin{definition}
A set $S\subset \R^n$ is called an \textbf{isolated invariant set} if it is the maximal invariant set in some isolating neighborhood.  That is, $S$ is an isolated invariant set if there is an isolating neighborhood $N$ such that $$S = \text{Inv}(N)$$
\end{definition}

\begin{lemma}
The closure of an invariant set is invariant.
\end{lemma}

\begin{proof}\label{inv_cls}
Assume the set $K$ is invariant.  Let $x$ be a limit point of $K$, so that there is some sequence $\{x_k\}_{k=1}^\infty\subset K$ converging to $x$.  Since $K$ is invariant, we have an associated sequence of orbits $\{\psi_k:\R\to K\}_{k=1}^\infty$ with $\psi_k(0)=x_k$.  By lemma \ref{extunisol} there is an orbit $\psi:\R\to \overline{K}$ where $\{\psi_k:\R\to K\}_{k=1}^\infty$ converges uniformly to $\psi$ on any compact interval.  Clearly $\psi(0)=x$.  Since the choice of $x\in\overline{K}$ was arbitrary, we see that the closure of $K$ is invariant.  
\end{proof}

\begin{corollary}
Isolated invariant sets are compact
\end{corollary}

\begin{proof}
Let $S$ be an isolated invariant set.  Since $S$ is contained in a compact neighborhood $N$ by assumption, the boundedness condition is immediate.  Since $S$ is contained in the interior of this neighborhood, its closure, $\overline{S}$, is also contained in the interior.  Since $\overline{S}$ is invariant by lemma \ref{inv_cls}, and $S=\Inv(N)$, we see that $S=\overline{S}$.  
\end{proof}

\section{Conclusions and Acknowledgements} 
\subsection{Conclusions and Future Work}

Showing that isolating neighborhoods are stable is the first step in a very long journey.  The natural next step is to decompose invariant sets into attractors, repellers, and a gradient-like region between them, and also to show that this decomposition continues under the sense of perturbation that we have defined here.  We have been able to complete these steps in \cite{thieme3}, available as a preprint on the arXiv.  But a lot of work still remains in generalizing Conley index theory to this setting.  For one, we still require a robust definition of the index itself.  Beyond that, we would also like to generalize the complete Morse decomposition to this setting, including showing that these systems possess a Lyapunov function.  However, other researchers have had success along these lines for slightly more limited classes of differential inclusions.  The index is generalized in \cite{mrozek}, \cite{dzedzej}, and \cite{casagrande}, and both \cite{kopanskii} and \cite{li} address the Morse decomposition and the Lyapunov function.  Though some technical challenges remain, these results make it seem highly plausible that Conley Index theory can be generalized to the setting we have discussed in this paper, allowing us to use it to analyze, among other things, any piecewise continuous vector field.

\subsection{Acknowledgements} 
We would like to thank all of the participants of the University of Minnesota's Mathematics of Climate Seminar for listening to an enormous number of presentations about these ideas.  Their comments and questions were crucial in the development of many of these ideas.  In particular, Richard McGehee--who runs the Climate Seminar--deserves a special acknowledgement; without his guidance this project would not have been possible.

Department of Mathematics

University of Minnesota

206 Church Street Se

Minneapolis, Minnesota, US

Email: {\it thiem019@umn.edu}

\end{document}